\documentclass[a4paper,10pt]{article}
\usepackage[pdftex,unicode]{hyperref}
\usepackage{wrapfig}
\usepackage{bm}
\usepackage[usenames]{color}

\usepackage[warn]{mathtext}
\usepackage[T1,T2A]{fontenc}
\usepackage[cp1251,utf8]{inputenc}
\usepackage[english]{babel}
\usepackage{amsmath}
\usepackage{amsfonts}
\usepackage{amssymb}
\usepackage{mathrsfs}
\usepackage{amsthm}
\usepackage{euscript}
\usepackage{bigints}
\usepackage{xcolor}
\usepackage[left=2cm,right=2cm,top=2cm,bottom=2cm,bindingoffset=0cm]{geometry}
\usepackage{graphicx}
\graphicspath{}
\DeclareGraphicsExtensions{.pdf,.png,.jpg}

\newtheorem{theorem}{Theorem}[section]

\theoremstyle{plain}\newtheorem*{MainT}{Theorem}
\theoremstyle{plain}\newtheorem{Le}{Лемма}[section]
\theoremstyle{plain}
\theoremstyle{definition}
\theoremstyle{plain}
\theoremstyle{plain}
\theoremstyle{plain}\newtheorem{Cor}[Le]{Corollary}
\theoremstyle{plain}
\theoremstyle{plain}
\theoremstyle{plain}
\theoremstyle{plain}
\theoremstyle{plain}
\theoremstyle{plain}
\numberwithin{equation}{section}

\newcommand{\R}{\mathbb{R}}

\newtheorem*{theorem*}{Theorem}{\bf}{\it}
\newtheorem*{proposition*}{Утверждение}{\bf}{\it}
\newtheorem*{observation*}{Наблюдение}{\bf}{\it}
\newtheorem*{lemma*}{Лемма}{\bf}{\it}
\theoremstyle{definition}

\theoremstyle{remark}

\newcommand{\BMO}{\mathrm{BMO}}
\newcommand{\Bel}{\mathbb{B}}
\newcommand{\eps}{\varepsilon}

\newcommand{\BG}{\mathcal{B}}

\author{Egor Dobronravov}
\title{A sharp symmetric integral form of the John--Nirenberg inequality\thanks{Supported by the Russian Science Foundation grant 19-71-10023}}

\begin{document}
\maketitle
\begin{abstract}
We find sharp constants in the symmetric integral form of the John--Nirenberg inequality. The result is based upon computation of a new interesting Bellman function.
\end{abstract}

\section{Introduction.}
The aim of this paper is to provide a sharp form of the John--Nirenberg inequality for functions of bounded mean oscillation of a single variable and to describe an interesting Bellman function related to this problem. Before we formulate the result, let us introduce the notation and describe the already known sharp versions of the John--Nirenberg inequality. %In the paper, we work with the space $\BMO$  that first appeared in~\cite{r4}. We equip $\BMO$ with the quadratic (semi-)norm.

For an interval $I$ and a function $\varphi\in L^1(I)$, let $\langle\varphi\rangle_I$ be the average value of $\varphi$ over $I$, that is $\langle\varphi\rangle_I=\frac{1}{|I|}\int_{I}\varphi$. Recall the definition of the~$\BMO$-(semi-)norm:
\begin{equation*}
\|\varphi\|_{\BMO(I)}^2
=
\sup\left\{
\left\langle\left(\varphi-\langle\varphi\rangle_J\right)^2\right\rangle_J\big|\,
J \text{ is a subinterval of } I
\right\}
,\qquad \varphi \in L^2(I).
\end{equation*}
The functions with finite~$\BMO(I)$-norm form the space~$\BMO(I)$. Let also
\begin{equation*}
\BMO_{\varepsilon}(I)
=
\left\{
\varphi\in \BMO(I)\big|\,\|\varphi\|_{\BMO}\leqslant\varepsilon
\right\},\qquad \eps > 0,
\end{equation*}
be the~$\eps$-ball of the~$\BMO$-space.
The fundamental John–Nirenberg inequality introduced in~\cite{r4} says that a~$\BMO$ function is exponentially integrable. There are several equivalent ways to express this principle. 

\begin{MainT}[Weak form]
There exist $c_1>0$ and $c_2>0$ such that the inequality holds
\begin{equation}
\frac{1}{|I|}\left|\left\{
s\in I\big|\,\varphi(s)-\langle\varphi\rangle_I\geqslant\lambda
\right\}\right|
\leqslant
c_1 \exp\left(\frac{-c_2\lambda}{\|\varphi\|_{\BMO}}\right)
\quad\text{for any }\varphi\in\BMO(I)
.
\end{equation}
for any $\lambda>0$.
\end{MainT}
Here and in what follows $|E|$ denotes the Lebesque measure of $E$, $E\subseteq\mathbb{R}$.
\begin{MainT}[Weak symmetric form]
There exist $c_1>0$ and $c_2>0$ such that the inequality holds
\begin{equation}
\frac{1}{|I|}\left|\left\{
s\in I\big|\,|\varphi(s)-\langle\varphi\rangle_I|\geqslant\lambda
\right\}\right|
\leqslant
c_1 \exp\left(\frac{-c_2\lambda}{\|\varphi\|_{\BMO}}\right)
\quad\text{for any }\varphi\in\BMO(I)
.
\end{equation}
for any $\lambda>0$.
\end{MainT}

\begin{MainT}[Integral form]
There exists $\varepsilon_0>0$ such that for every $0\leqslant\varepsilon<\varepsilon_0$ there is $C(\varepsilon)>0$ such that for any function $\varphi\in\BMO_{\varepsilon}(I)$
\begin{equation}
\langle e^{\varphi}\rangle_I\leqslant C(\varepsilon)e^{\langle\varphi\rangle_I}
,
\end{equation}
or, equivalently,
\begin{equation}
\langle \exp(\varphi-\langle\varphi\rangle_I)\rangle_I\leqslant C(\varepsilon)
.
\end{equation}
\end{MainT}

\begin{MainT}[Integral symmetric form]
There exists $\varepsilon_0>0$ such that for every $0\leqslant\varepsilon<\varepsilon_0$ there is $C(\varepsilon)>0$ such that for any function $\varphi\in\BMO_{\varepsilon}(I)$,
\begin{equation}
\langle \exp|\varphi-\langle\varphi\rangle_I|\rangle_I\leqslant C(\varepsilon)
.
\end{equation}
\end{MainT}
These four theorems are equivalent on the qualitative level. However, if one is interested in sharp values of the parameters in these inequalities, then the equivalence is at least non-trivial (most likely non-existent). 
The sharp forms of the first two theorems and the Bellman functions 
\begin{equation}
{\bf B}_{\max}(x_1,\,x_2,\,\lambda,\,\varepsilon)
=
\frac{1}{|I|}
\sup\left\{
\left|\left\{
s\in I\big|\,\varphi(s)\geqslant\lambda
\right\}\right|\colon \varphi\in\BMO_{\varepsilon}(I),\,\langle \varphi\rangle=x_1,\,\langle \varphi^2\rangle=x_2
\right\}
;
\end{equation}
\begin{equation}
{\bf B}(x_1,\,x_2,\,\lambda,\,\varepsilon)
=
\frac{1}{|I|}
\sup\left\{
\left|\left\{
s\in I\big|\,|\varphi(s)|\geqslant\lambda
\right\}\right|\colon \varphi\in\BMO_{\varepsilon}(I),\,\langle \varphi\rangle=x_1,\,\langle \varphi^2\rangle=x_2
\right\}
.
\end{equation}
were calculated in~\cite{r2}.
So, for $\varphi\in\BMO_{\varepsilon}$ we can write
\begin{equation}
\frac{1}{|I|}\left|\left\{
s\in I\big|\,\varphi(s)-\langle\varphi\rangle_I\geqslant\lambda
\right\}\right|
\leqslant
C(\varepsilon,\,\lambda)
;
\end{equation}
\begin{equation}
\frac{1}{|I|}\left|\left\{
s\in I\big|\,|\varphi(s)-\langle\varphi\rangle_I|\geqslant\lambda
\right\}\right|
\leqslant
C_{sym}(\varepsilon,\,\lambda)
,
\end{equation}
where the sharp value, of $C_{\max}$ and $C$ are calculated from the corresponding Bellman functions in the following way
\begin{equation}
C(\varepsilon,\,\lambda)
=
\sup
\left\{
{\bf B}_{\max}(0,\,x_2,\,\lambda,\,\varepsilon)\big|
x_2\in[0,\varepsilon^2]
\right\}
=
\begin{cases}
1-\frac{\lambda}{2\varepsilon},  & \quad 0\leqslant\lambda\leqslant\varepsilon,   \\
\frac{e}{2}e^{-\frac{\lambda}{\varepsilon}},&  \quad \varepsilon\leqslant\lambda;
\end{cases}
\end{equation}
\begin{equation}
C_{sym}(\varepsilon,\,\lambda)
=
\sup
\left\{
{\bf B}(0,\,x_2,\,\lambda,\,\varepsilon)\big|
x_2\in[0,\varepsilon^2]
\right\}
=
\begin{cases}
1,  & \quad 0\leqslant\lambda\leqslant\varepsilon,   \\
\frac{\varepsilon^2}{\lambda^2},  & \quad \varepsilon\leqslant\lambda\leqslant 2\varepsilon,   \\
\frac{e^2}{4}e^{-\frac{\lambda}{\varepsilon}},&  \quad 2\varepsilon\leqslant\lambda.
\end{cases}
\end{equation}
From these formulas it follows that the optimal $c_2$ for both first and second theorems is equal to $1$. If $c_2=1$, then the optimal $c_1$ is equal to $\frac{e}{2}$ and $e$ for the first and second theorems respectively. The sharp form of the third theorem and the corresponding Bellman function were calculated in~\cite{r3}.
More precisely, the Bellman function given by~\eqref{nonsymbel} was calculated in~\cite{r3}. It followed, that $\varepsilon_0$ is equal to 1, and the optimal $C(\varepsilon)$ for the third theorem is equal to $\frac{e^{-\varepsilon}}{1-\varepsilon}$.
In this paper, we will compute the sharp values of~$\eps_0$ and~$C(\eps)$ in the fourth theorem (see Theorem~\ref{Theorem} below).

We should note that our way to define the~$\BMO$-norm is not the most common one, at least in the field of real analysis. A version based on the~$L^1$ norm instead of~$L^2$ is more widespread. The two semi-norms are equivalent by the John--Nirenberg inequality. See the papers~\cite{Kor} and~\cite{Lerner} for information about sharp John--Nirenberg inequalities with the~$L^1$-based norm and~\cite{Slavin15} and~\cite{SlavinVasyunin} for the inequalities with the~$L^p$-based norms.

We refer the reader to~\cite{r12} and~\cite{r14} for the exposition of the Bellman function method and to the papers~\cite{r7} and~\cite{r6} for a general treatment of more specific $\BMO$-based Bellman function problems. 
\section{Main result and the Bellman setup.}
\begin{theorem}\label{Theorem}
Let $I\subset\mathbb{R}$ be an interval and $\varphi\in\BMO_{\varepsilon}(I)$. The inequality
\begin{equation}\label{equation21}
\langle
\exp |\varphi-\langle\varphi\rangle_I|
\rangle_I
\leqslant
\begin{cases}
e^{\varepsilon},  & \quad 0\leqslant\varepsilon\leqslant\frac{1}{2},   \\
\dfrac{e^{1-\varepsilon}}{2-2\varepsilon}, &\quad\frac{1}{2}\leqslant\varepsilon < 1,\\
+\infty,&  \quad\varepsilon \geqslant 1
\end{cases}
\end{equation}
holds true. It is sharp and attainable for all $\varepsilon\in [0,\,1)$.
\end{theorem}
Using the techniques of~\cite{StolyarovZatitskiy}, one may transfer this sharp inequality to the line. Let~$\varphi$ be a locally square summable function on the line, let
\begin{equation*}
\|\varphi\|_{\BMO(\R)}
=
\sup\left\{
\left(\left\langle\left(\varphi-\langle\varphi\rangle_J\right)^2\right\rangle_J\right)^{\frac{1}{2}}\big|\,
J \text{ is an interval } 
\right\}.
\end{equation*}
For $\varepsilon\geqslant 0$, let~$\BMO_{\varepsilon}(\R)
=
\left\{
\varphi\in \BMO(\mathbb{R})\big|\,\|\varphi\|_{\BMO}\leqslant\varepsilon
\right\}$.
\begin{Cor}
Let $\varphi\in\BMO_{\varepsilon}(\R)$. The inequality~\eqref{equation21} holds true. It is sharp for all $\varepsilon\in [0,\,1)$.
\end{Cor}

Theorem~\ref{Theorem} will follow from the formula for a specific Bellman function we will present below, see~\eqref{equation28} and Theorem~\ref{Theorem1}. 
Let $P=\left\{(x_1,x_2)\big|\,x_2=x_1^2\right\}$ be the standard parabola, let $\Omega_{\varepsilon}=\left\{(x_1,x_2)\big|\,x_1^2\leqslant x_2\leqslant x_1^2+\varepsilon^2\right\}$ be the parabolic strip of width $\varepsilon^2$. We will consider summable functions that map~$I$ to~$P$. Let
\begin{equation*}
U_{\varepsilon}=
\left\{
\varphi\in L^1(I,\,P)\big|\, \forall \text{ interval }J\subseteq I\ \ \langle\varphi\rangle_J\in\Omega_{\varepsilon}
\right\}
.
\end{equation*}
Then the mapping $T\colon L^2(I,\,\mathbb{R})\to L^1(I,\,P)\ $ given by the formula $\varphi\mapsto \left(\varphi,\,\varphi^2\right)$ defines a bijection between $\BMO_{\varepsilon}$ and $U_{\varepsilon}$. We define the Bellman function~$\Bel_{\varepsilon}\colon \Omega_{\varepsilon}\to \mathbb{R}\cup\{+\infty\}$ by the formula
\begin{equation}\label{BellmanDefinition}
\Bel_{\varepsilon}(x_1,\,x_2)
=
\sup\left\{
\left\langle e^{|\varphi|}\right\rangle_I\big|\, \varphi\in\BMO_{\varepsilon}(I),\,\langle\varphi\rangle_I=x_1,\,\left\langle\varphi^2\right\rangle_I=x_2
\right\}
.
\end{equation}
Our aim is to find a good analytic expression for $\Bel_{\varepsilon}$. The papers~\cite{r7} and~\cite{r6} provide a formula for a more general function
\begin{equation}\label{BellmanDefinition}
\Bel_{f,\varepsilon}(x_1,\,x_2)
=
\sup\left\{
\left\langle f(\varphi)\right\rangle_I\big|\, \varphi\in\BMO_{\varepsilon}(I),\,\langle\varphi\rangle_I=x_1,\,\left\langle\varphi^2\right\rangle_I=x_2
\right\},
\end{equation}
where~$f$ is an arbitrary~$C^{3}$-smooth function satisfying additional regularity assumptions. Our Theorem~\ref{Theorem1} below is not covered by these results since in our case~$f(t) = e^{|t|}$, and this function is far from being~$C^3$ or even~$C^2$ smooth (the exact smoothness condition in the cited papers is slightly less than~$C^3$, but definitely stronger than~$C^2$). In fact, the non-smoothness of the cost function $f$ will lead to some unexpected effects. What is more, the papers~\cite{r7} and~\cite{r6} do not provide a simple formula for the Bellman function. They suggest an algorithm for composing such a formula. Our cost function allows several shortcuts in the general algorithm that are also interesting in themselves. 

We say that a function with the domain $V\subseteq\mathbb{R}^d$ is locally concave if its restriction to any segment that lies in $V$, is concave. We can also define the class of locally concave functions
\begin{equation*}
\Lambda_{\varepsilon}
=
\left\{
B\colon\Omega_{\varepsilon}\to\mathbb{R}\cup\{+\infty\}
\big|\, 
B\text{ is locally concave and }B\left(x_1,\,x_1^2\right)\geqslant e^{|x_1|}, \ x_1 \in \mathbb{R}
\right\}
\end{equation*}
and the pointwise minimal function~$\BG_{\varepsilon}\colon \Omega_{\varepsilon}\to \mathbb{R}\cup\{+\infty\}$,
\begin{equation*}
\BG_{\varepsilon}(x_1,\,x_2)
=
\inf\left\{
B(x_1,\,x_2)
\big|\,
B\in\Lambda_{\varepsilon}
\right\}
.
\end{equation*}
One may see that $\BG_{\varepsilon}\in\Lambda_{\varepsilon}$. According to the main theorem of~\cite{r5}, $\Bel_{\varepsilon}=\BG_{\varepsilon}$. To describe~$\Bel_\eps$, we will need the auxiliary parameters $\beta,\,\alpha,\,c,\,b$:
\begin{equation}\label{beta}
\beta=\sqrt{\varepsilon^2+x_1^2-x_2},
\end{equation}
\begin{equation}\label{alpha}
\alpha=\varepsilon+\frac{\varepsilon}{1+\varepsilon}\left(\text{log}(1-\varepsilon)-\text{log}\left(2\varepsilon^2\right)\right),
\end{equation}
\begin{equation}\label{c}
c=\frac{1-\alpha}{1-\varepsilon^2}e^{\alpha}-\frac{\varepsilon+\alpha}{2\varepsilon(1+\varepsilon)}\exp\left\{-\frac{1}{\varepsilon}(\alpha-\varepsilon)+\varepsilon\right\},
\end{equation}
\begin{equation}\label{d}
b=\frac{1}{2\left(1-\varepsilon^2\right)}e^{\alpha}+\frac{1}{4\varepsilon(1+\varepsilon)}\exp\left\{-\frac{1}{\varepsilon}(\alpha-\varepsilon)+\varepsilon\right\}.
\end{equation}

For $0\leqslant\varepsilon<\frac{1}{2}$, we split the domain $\Omega_{\varepsilon}$ into four parts (see Fig.~\ref{r7})$\colon$
\begin{equation*}
\Omega_{\varepsilon}^1=\left\{(x_1,x_2)\in\Omega_{\varepsilon}\big|\, x_2\leqslant\varepsilon^2\right\},
\end{equation*}
\begin{equation*}
\Omega_{\varepsilon}^2
=
\big\{(x_1,x_2)\in\Omega_{\varepsilon}\big|
 x_2\in\left[\varepsilon^2,\,\alpha^2\right]\text{ and }x_2\leqslant 2\alpha|x_1|-\alpha^2-2\varepsilon(|x_1|-\alpha)\text{ when }|x_1|>\alpha-\varepsilon\big\}
,
\end{equation*}
\begin{equation*}
\Omega_{\varepsilon}^3=\left\{(x_1,x_2)\in\Omega_{\varepsilon}\big|\,|x_1|\in[\alpha-\varepsilon,\,\alpha+\varepsilon],\,x_2\geqslant2\alpha|x_1|-\alpha^2+\left|2\varepsilon(|x_1|-\alpha)\right|\right\},
\end{equation*}
\begin{equation*}
\Omega_{\varepsilon}^4=\left\{(x_1,x_2)\in\Omega_{\varepsilon}\big|\,|x_1|\geqslant\alpha\text{ and }x_2\leqslant 2(\alpha+\varepsilon)|x_1|-\alpha^2-2\varepsilon\alpha\text{ when }|x_1|<\alpha+\varepsilon\right\},
\end{equation*}
and define the candidate $B_{\varepsilon}$ by the formula$\colon$
\begin{equation}\label{equation28}
B_{\varepsilon}(x_1,\,x_2)
=
\begin{cases}
e^{\sqrt{x_2}}, &\text{for }(x_1,\,x_2)\in\Omega_{\varepsilon}^1,   \\
\frac{1+\beta}{1+\varepsilon}\exp\{|x_1|-\beta+\varepsilon\}+\frac{\varepsilon-\beta}{1+\varepsilon}\exp\left\{-\frac{1}{\varepsilon}(|x_1|-\beta)+\varepsilon\right\}, & \text{for }(x_1,\,x_2)\in\Omega_{\varepsilon}^2,   \\
c(|x_1|-\alpha)+b\left(x_2-\alpha^2\right)+e^{\alpha}, & \text{for } (x_1,\,x_2)\in\Omega_{\varepsilon}^3, \\   
\frac{1-\beta}{1-\varepsilon}\exp\{|x_1|+\beta-\varepsilon\}, & \text{for }(x_1,\,x_2)\in\Omega_{\varepsilon}^4.
\end{cases}
\end{equation} 
\begin{figure}[h]
\center{\includegraphics[scale=0.45]{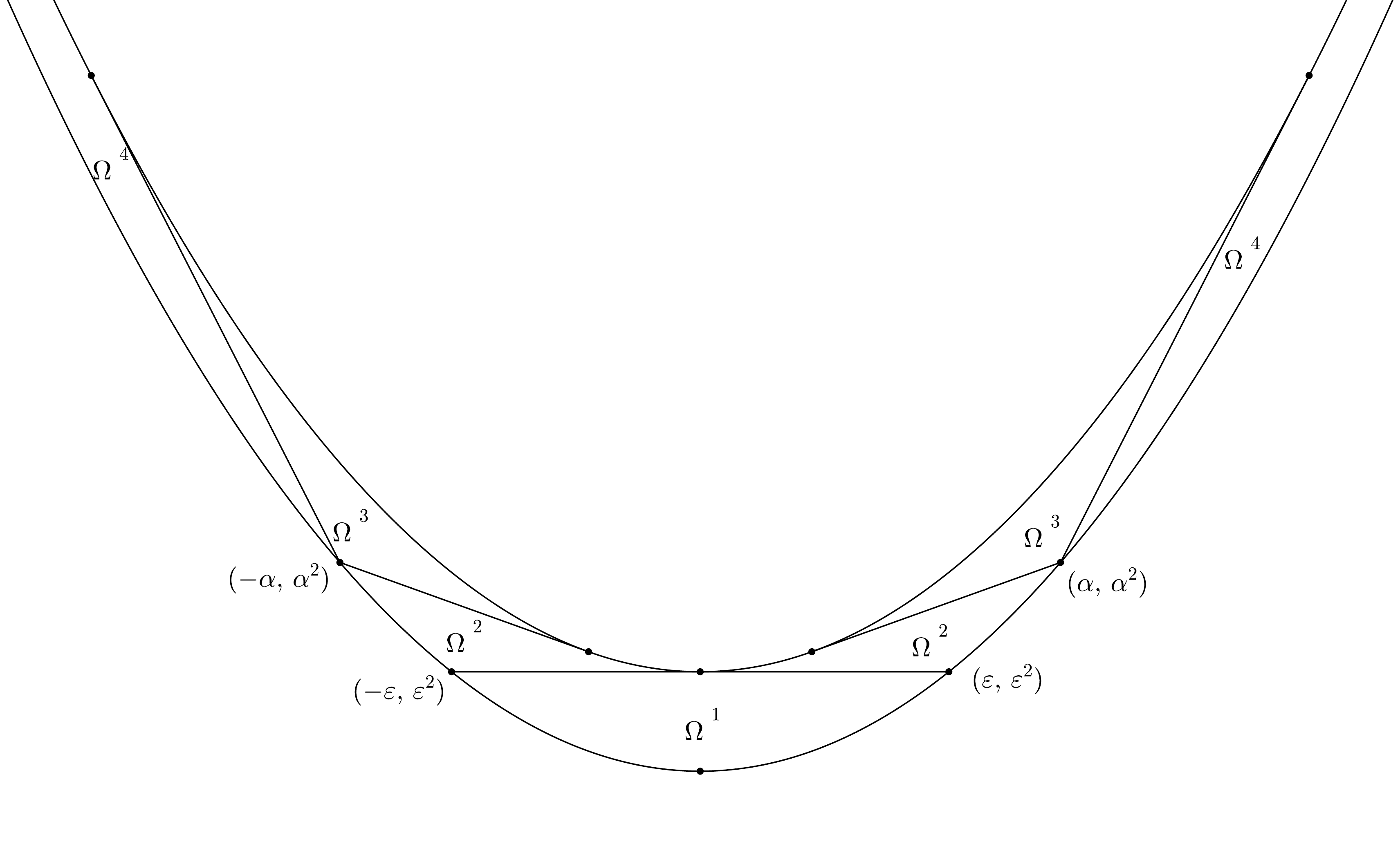}}
\caption{Foliation in the case $\varepsilon\in\left[0,\,\frac{1}{2}\right)$.}\label{r7}
\end{figure}
For $\frac{1}{2}\leqslant\varepsilon<1$,  we split the domain $\Omega_{\varepsilon}$ into three parts (see Fig.~\ref{r11})$\colon$
\begin{equation*}
\Omega_{\varepsilon}^1=\left\{(x_1,\,x_2)\in\Omega_{\varepsilon}\big|\,x_2\leqslant(1-\varepsilon)^2\right\},
\end{equation*}
\begin{equation*}
\Omega_{\varepsilon}^2=\left\{(x_1,\,x_2)\in\Omega_{\varepsilon}\big|\,x_2\in\left[(1-\varepsilon)^2,\,1+\varepsilon^2\right]\text{ and }|x_1|\leqslant\frac{1+x_2-\varepsilon^2}{2}\right\},
\end{equation*}
\begin{equation*}
\Omega_{\varepsilon}^3=\left\{(x_1,\,x_2)\in\Omega_{\varepsilon}\big|\,x_2\geqslant(1-\varepsilon)^2\text{ and }|x_1|\geqslant\frac{1+x_2-\varepsilon^2}{2}\text{ when }x_2<1+\varepsilon^2\right\},
\end{equation*}
and define the candidate $B_{\varepsilon}$ by the formula$\colon$
\begin{equation}\label{equation2}
B_{\varepsilon}(x_1,\,x_2)=
\begin{cases}
e^{\sqrt{x_2}},& \text{for }(x_1,\,x_2)\in\Omega_{\varepsilon}^1,   \\
\frac{1-\varepsilon^2+x_2}{2-2\varepsilon}e^{1-\varepsilon},  &\text{for }(x_1,\,x_2)\in\Omega_{\varepsilon}^2,   \\
\frac{1-\beta}{1-\varepsilon}\exp\left(x_1+\beta-\varepsilon\right),& \text{for }(x_1,\,x_2)\in\Omega_{\varepsilon}^3.
\end{cases}
\end{equation}
\begin{figure}[h]
\center{\includegraphics[scale=0.5]{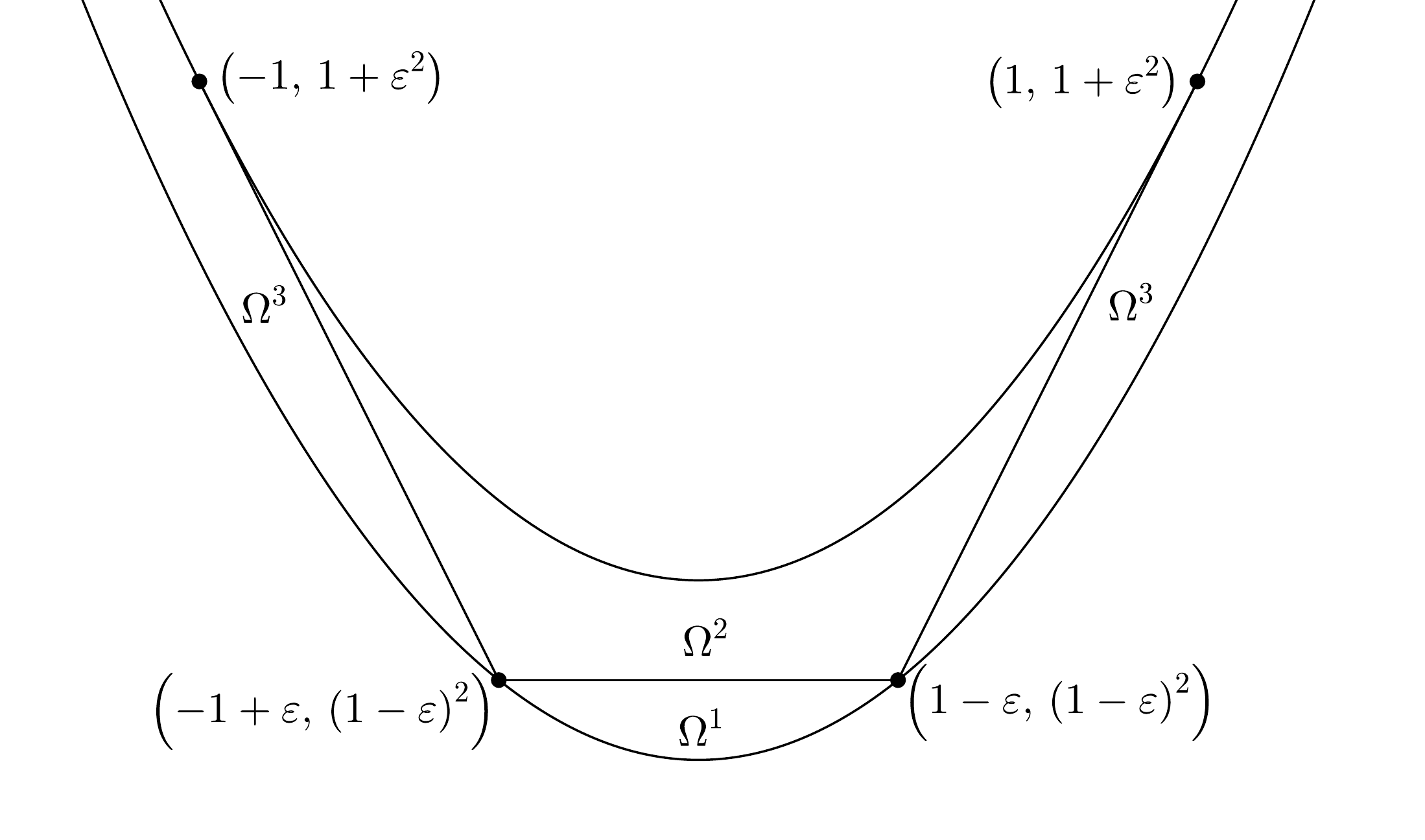}}
\caption{Foliation in the case $\varepsilon\in\left[\frac{1}{2},\,1\right)$.}\label{r11}
\end{figure}

Finally for $\varepsilon\geqslant 1$ we set$\colon$
\begin{equation}
B_{\varepsilon}(x_1,\,x_2)=
\begin{cases}
e^{|x_1|},& \text{for }x_2=x_1^2,   \\
+\infty,&\text{otherwise.}   
\end{cases}
\end{equation}
\begin{theorem}\label{Theorem1}
For $\varepsilon\geqslant 0$, we have $\Bel_{\varepsilon}=B_{\varepsilon}$. Moreover, for all $\varepsilon\geqslant 0$ and $(x_1,\,x_2)\in\Omega_{\varepsilon}$ there exists a function $\varphi\in\BMO_{\varepsilon}$ such that $\langle\varphi\rangle_I=x_1,\,\left\langle\varphi^2\right\rangle_I=x_2$, and $\langle\exp|\varphi|\rangle_I=\Bel_{\varepsilon}(x_1,\,x_2)$.
\end{theorem}
To prove Theorem~\ref{Theorem1}, we only need to construct the functions $\varphi\in\BMO_{\varepsilon}$ with corresponding means: $(\langle\varphi\rangle_I,\,\langle\varphi^2\rangle_I)=(x_1,\, x_2)$ such that $B_{\varepsilon}(x_1,\,x_2)=\left\langle e^{|\varphi|}\right\rangle_I$, and check the local concavity of $B_{\varepsilon}$. The construction of such a function $\varphi$ provides the inequality
\begin{equation*}
B_{\varepsilon}(x_1,\,x_2)
=
\left\langle e^{|\varphi|}\right\rangle_I
\leqslant
\sup\left\{\langle\exp\{ |\psi|\}\rangle_I\big|\,\psi\in\BMO_{\varepsilon}\text{ and }(\langle\psi\rangle_I,\,\langle\psi^2\rangle_I)=(x_1,\,x_2)\right\}
=
\Bel_{\varepsilon}(x_1,\,x_2)
.
\end{equation*}
If we check the local concavity of $B_{\varepsilon}$, then from $B_{\varepsilon}(x_1,\,x_1^2)=e^{|x_1|}$ we will have
\begin{equation*}
\Bel_{\varepsilon}(x_1,\,x_2)
=
\BG_{\varepsilon}(x_1,\,x_2)
=
\inf\left\{B(x_1,\,x_2)\big|\,B\in\Lambda_{\varepsilon}\right\}
\leqslant
B_{\varepsilon}(x_1,\,x_2)
.
\end{equation*}
Thus, we will obtain $B_{\varepsilon}=\Bel_{\varepsilon}$. We construct the desired extremal functions $\varphi$ in Section~\ref{attainability}, and check the local concavity in Section~\ref{concavity}. In Sections~\ref{continuity} and~\ref{differentiability}, we verify that $B_{\varepsilon}\in C(\Omega_{\varepsilon})$ and $B_{\varepsilon}\in C^1(\Omega_{\varepsilon}\setminus\{(0,\,0)\})$ respectively, these facts are used in Sections~\ref{concavity} and~\ref{attainability}.
\begin{proof}[Proof of Theorem~\ref{Theorem} based on Theorem~\ref{Theorem1}]
The following equality holds$\colon$
\begin{multline*}
\sup\left\{\exp\{ |\varphi-\langle\varphi\rangle_I|\}\big|\,\varphi\in\BMO_{\varepsilon}\right\}
=
\sup\left\{\exp\{ |\varphi|\}\big|\,\varphi\in\BMO_{\varepsilon}\text{ and }\langle\varphi\rangle_I=0\right\}
=
\\
=
\sup\left\{\Bel_{\varepsilon}(0,\,x_2)\big|\,x_2\in\left[0,\,\varepsilon^2\right]\right\}
=
\Bel_{\varepsilon}\left(0,\,\varepsilon^2\right)
=
\begin{cases}
e^{\varepsilon},  & \text{for } 0\leqslant\varepsilon\leqslant\frac{1}{2},   \\
\dfrac{e^{1-\varepsilon}}{2-2\varepsilon}, &\text{for }\frac{1}{2}\leqslant\varepsilon < 1,\\
+\infty,&  \text{for }\varepsilon \geqslant 1.
\end{cases}
\end{multline*}
The attainability follows from the existence of the function $\varphi\in\BMO_{\varepsilon}$ such that $\langle\varphi\rangle_I=0,\,\left\langle\varphi^2\right\rangle_I=\varepsilon^2$ and $\langle\exp\{|\varphi|\}\rangle_I=\Bel_{\varepsilon}\left(0,\,\varepsilon^2\right)$.
\end{proof}
Let also $B_{\varepsilon}^{asym}$ be the Bellman function for the non-symmetric integral John-Nirenberg inequality$\colon$
\begin{equation}\label{nonsymbel}
B_{\varepsilon}^{asym}(x_1,\,x_2)
=
\sup\left\{
\left\langle e^{\varphi}\right\rangle_I\big|\, \varphi\in\BMO_{\varepsilon}(I),\,\langle\varphi\rangle_I=x_1,\,\left\langle\varphi^2\right\rangle_I=x_2
\right\},\qquad x\in \Omega_\eps
.
\end{equation}
The function $B_{\varepsilon}^{asym}$ was calculated in~\cite{r3}, and it has the form

\begin{equation}\label{asymbel}
B_{\varepsilon}^{asym}(x_1,\,x_2)
=
\frac{1-\sqrt{\varepsilon^2+x_1^2-x_2}}{1-\varepsilon}\exp\left(x_1+\sqrt{\varepsilon^2+x_1^2-x_2}-\varepsilon\right)
,
\ \ \ \ 
\varepsilon\in [0,\,1),
\end{equation}
\begin{equation}
B_{\varepsilon}^{asym}(x_1,\,x_2)=
\begin{cases}
e^{x_1},& \text{for }x_2=x_1^2,   \\
+\infty,&\text{otherwise}   
\end{cases}
\ \ \ \ \varepsilon\in[1,\,+\infty).
\end{equation}

\section{Continuity of $B_{\varepsilon}$.}\label{continuity}
In the section, we will verify that $B_{\varepsilon}$ is continuous provided $\varepsilon\in(0,\,1)$. From~\eqref{equation28} and~\eqref{equation2} we see that $B_{\varepsilon}$ is a piecewise continuous function. We only need to check the continuity of gluing between $\Omega_{\varepsilon}^j$ and $\Omega_{\varepsilon}^{j+1}$. By symmetry we may assume that $x_1\geqslant 0$.
\subsection{The case $\varepsilon\in\left(0,\,\frac{1}{2}\right)$.}
Let us introduce the functions $g_1,\,g_2,\,g_3,\,g_4\colon$
\begin{equation}
g_1(x_1,\,x_2)=e^{\sqrt{x_2}},\ \ \ \ x=(x_1,\,x_2)\in\Omega_{\varepsilon}^1,
\end{equation}
\begin{equation}
g_2(x_1,\,x_2)=\frac{1+\beta}{1+\varepsilon}\exp\{|x_1|-\beta+\varepsilon\}+\frac{\varepsilon-\beta}{1+\varepsilon}\exp\left\{-\frac{1}{\varepsilon}(|x_1|-\beta)+\varepsilon\right\},\ \ \ \ x=(x_1,\,x_2)\in\Omega_{\varepsilon}^2,
\end{equation}
\begin{equation}
g_3(x_1,\,x_2)=c(|x_1|-\alpha)+b\left(x_2-\alpha^2\right)+e^{\alpha},\ \ \ \ x=(x_1,\,x_2)\in\Omega_{\varepsilon}^3,
\end{equation}
\begin{equation}
g_4(x_1,\,x_2)=\frac{1-\beta}{1-\varepsilon}\exp\{|x_1|+\beta-\varepsilon\},\ \ \ \ x=(x_1,\,x_2)\in\Omega_{\varepsilon}^4,
\end{equation}
where $\alpha,\,\beta,\,c,\,b$ are defined in~\eqref{alpha},~\eqref{beta},~\eqref{c},~\eqref{d} respectively. Then we only need to check that $g_j=g_{j+1}$ on the $\Omega_{\varepsilon}^j\cap\Omega_{\varepsilon}^{j+1}$.

\subsubsection{Continuity on $\Omega_{\varepsilon}^1\cap\Omega_{\varepsilon}^2$.}
The set $\Omega_{\varepsilon}^1\cap\Omega_{\varepsilon}^2$ is the segment $\left[\left(-\varepsilon,\,\varepsilon^2\right),\,\left(\varepsilon,\,\varepsilon^2\right)\right]$, see Fig.~\ref{r7}. On this segment, we have $\beta=|x_1|$, $x_2=\varepsilon^2$ hence
\begin{equation*}
g_1(x_1,\,x_2)
=
e^{\varepsilon}
=
g_2(x_1,\,x_2).
\end{equation*}
\subsubsection{Continuity on $\Omega_{\varepsilon}^2\cap\Omega_{\varepsilon}^3$.}
The set $\Omega_{\varepsilon}^2\cap\Omega_{\varepsilon}^3\cap\{x_1\geqslant 0\}$ is the segment 
\begin{equation}\label{seg}
\left[\left(\alpha-\varepsilon,\,(\alpha-\varepsilon)^2+\varepsilon^2\right),\,\left(\alpha,\,\alpha^2\right)\right].
\end{equation}
The identity $x_1-\beta=u-\varepsilon$ holds on the segment $\left[\left(u-\varepsilon,\,(u-\varepsilon)^2+\varepsilon^2\right),\,\left(u,\,u^2\right)\right]$ for any $u\geqslant\varepsilon$. Hence $g_2$ is linear on~\eqref{seg}, and $g_3$ is also linear on this segment. We need to verify the coincidence of $g_2$ and $g_3$ only at the endpoints. The endpoint $(\alpha,\,\alpha^2)$ (note that $\beta=\varepsilon$ at this point)$\colon$
\begin{equation*}
g_2\left(\alpha,\,\alpha^2\right)
=
\frac{1+\beta}{1+\varepsilon}\exp\{\alpha-\varepsilon+\varepsilon\}+\frac{\varepsilon-\beta}{1+\varepsilon}\exp\left\{-\frac{1}{\varepsilon}(\alpha-\varepsilon)+\varepsilon\right\}
=
e^{\alpha}
=
g_3\left(\alpha,\,\alpha^2\right)
.
\end{equation*}
Here the verification for the second endpoint ($\beta=0$ at this point) is$\colon$
\begin{equation*}
g_2\left(\alpha-\varepsilon,\,(\alpha-\varepsilon)^2+\varepsilon^2\right)
=
\frac{1}{1+\varepsilon}\exp\{\alpha\}+\frac{\varepsilon}{1+\varepsilon}\exp\left\{-\frac{1}{\varepsilon}(\alpha-\varepsilon)+\varepsilon\right\}
,
\end{equation*}
\begin{multline*}
g_3\left(\alpha-\varepsilon,\,(\alpha-\varepsilon)^2+\varepsilon^2\right)
=
-c\varepsilon+b\left(2\varepsilon^2-2\varepsilon\alpha\right)+e^{\alpha}
\overset{\eqref{c},\eqref{d}}{=}
\\
=
e^{\alpha}\left(1-\frac{\varepsilon(1-\alpha)}{1-\varepsilon^2}+\frac{\varepsilon^2-\varepsilon\alpha}{1-\varepsilon^2}\right)+
\exp\left(-\frac{1}{\varepsilon}(\alpha-\varepsilon)+\varepsilon\right)\left(\frac{\varepsilon+\alpha}{2(1+\varepsilon)}+\frac{\varepsilon-\alpha}{2(1+\varepsilon)}\right)
=
g_2\left(\alpha-\varepsilon,\,(\alpha-\varepsilon)^2+\varepsilon^2\right).
\end{multline*}

\subsubsection{Continuity on $\Omega_{\varepsilon}^3\cap\Omega_{\varepsilon}^4$.}
The set $\Omega_{\varepsilon}^3\cap\Omega_{\varepsilon}^4\cap\{x_1\geqslant 0\}$ is the segment 
\begin{equation}\label{seg2}
\left[\left(\alpha,\,\alpha^2\right),\,\left(\alpha+\varepsilon,\,(\alpha+\varepsilon)^2+\varepsilon^2\right)\right].
\end{equation}
The identity $x_1+\beta=u+\varepsilon$ holds on the segment $\left[\left(u,\,u^2\right),\,\left(u+\varepsilon,\,(u+\varepsilon)^2+\varepsilon^2\right)\right]$ for any $u\geqslant 0$. Hence $g_4$ is linear on~\eqref{seg2}, and $g_3$ is also linear in this segment. We need to check the coincidence of $g_3$ and $g_4$ only at the endpoints. The endpoint $(\alpha,\,\alpha^2)$ is simple$\colon$
\begin{equation*}
g_4\left(\alpha,\,\alpha^2\right)
=
\frac{1-\beta}{1-\varepsilon}\exp\{\alpha+\varepsilon-\varepsilon\}
=
e^{\alpha}
=
g_3\left(\alpha,\,\alpha^2\right)
.
\end{equation*}
Here the verification for the second endpoint is$\colon$
\begin{equation*}
g_4\left(\alpha+\varepsilon,\,(\alpha+\varepsilon)^2+\varepsilon^2\right)
=
\frac{1-\beta}{1-\varepsilon}\exp\{\alpha+\varepsilon-\varepsilon\}
=
\frac{1}{1-\varepsilon}e^{\alpha}
,
\end{equation*}
\begin{multline*}
g_3\left(\alpha+\varepsilon,\,(\alpha+\varepsilon)^2+\varepsilon^2\right)
=
c\varepsilon+b\left(2\varepsilon^2+2\varepsilon\alpha\right)+e^{\alpha}
=
\\
=
e^{\alpha}\left(1+\frac{\varepsilon(1-\alpha)}{1-\varepsilon^2}+\frac{\varepsilon^2+\varepsilon\alpha}{1-\varepsilon^2}\right)+
\exp\left(-\frac{1}{\varepsilon}(\alpha-\varepsilon)+\varepsilon\right)\left(-\frac{\varepsilon+\alpha}{2(1+\varepsilon)}+\frac{\varepsilon+\alpha}{2(1+\varepsilon)}\right)
=
g_4\left(\alpha+\varepsilon,\,(\alpha+\varepsilon)^2+\varepsilon^2\right).
\end{multline*}

\subsection{The case $\varepsilon\in\left[\frac{1}{2},\,1\right)$.}
Let us introduce the functions $g_1,\,g_2,\,g_3\colon$
\begin{equation}
g_1(x_1,\,x_2)=e^{\sqrt{x_2}},\ \ \ \ x=(x_1,\,x_2)\in\Omega_{\varepsilon}^1,
\end{equation}
\begin{equation}
 g_2(x_1,\,x_2)=\frac{1-\varepsilon^2+x_2}{2-2\varepsilon}e^{1-\varepsilon},\ \ \ \ x=(x_1,\,x_2)\in\Omega_{\varepsilon}^2,
\end{equation}
\begin{equation}
g_3(x_1,\,x_2)=\frac{1-\sqrt{\varepsilon^2+x_1^2-x_2}}{1-\varepsilon}\exp\left(x_1+\sqrt{\varepsilon^2+x_1^2-x_2}-\varepsilon\right),\ \ \ \ x=(x_1,\,x_2)\in\Omega_{\varepsilon}^3.
\end{equation}
Then we only need to check that $g_j=g_{j+1}$ on the $\Omega_{\varepsilon}^j\cap\Omega_{\varepsilon}^{j+1}$.

\subsubsection{Continuity on $\Omega_{\varepsilon}^1\cap\Omega_{\varepsilon}^2$.}
The set $\Omega_{\varepsilon}^1\cap\Omega_{\varepsilon}^2$ is the segment $\left[\left(-1+\varepsilon,\,(1-\varepsilon)^2\right),\,\left(1-\varepsilon,\,(1-\varepsilon)^2\right)\right]$. On this segment, we have 
\begin{equation*}
g_1(x_1,\,x_2)
=
e^{1-\varepsilon}
=
\frac{1-\varepsilon^2+(1-\varepsilon)^2}{2-2\varepsilon}e^{1-\varepsilon}
=
g_2(x_1,\,x_2)
.
\end{equation*}
\subsubsection{Continuity on $\Omega_{\varepsilon}^2\cap\Omega_{\varepsilon}^3$.}
The set $\Omega_{\varepsilon}^2\cap\Omega_{\varepsilon}^3\cap\{x_1\geqslant 0\}$ is the segment 
\begin{equation}\label{seg3}
\left[\left(1-\varepsilon,\,(1-\varepsilon)^2\right),\,\left(1,\,1+\varepsilon^2\right)\right]
.
\end{equation}
The identity $x_1+\sqrt{\varepsilon^2+x_1^2-x_2}=u+\varepsilon$ holds on the segment $\left[\left(u,\,u^2\right),\,\left(u+\varepsilon,\,(u+\varepsilon)^2+\varepsilon^2\right)\right]$ for any $u\geqslant 0$. Hence $g_3$ is linear on~\eqref{seg3}, and $g_2$ is also linear on the segment. We need to check the coincidence of $g_2$ and $g_3$ only at the endpoints. Here the verification is$\colon$
\begin{equation*}
g_2\left(1-\varepsilon,\,(1-\varepsilon)^2\right)
=
e^{1-\varepsilon}
=
g_3\left(1-\varepsilon,\,(1-\varepsilon)^2\right)
,
\end{equation*}
\begin{equation*}
g_2\left(1,\,1+\varepsilon^2\right)
=
\frac{1}{1-\varepsilon}e^{1-\varepsilon}
=
g_3\left(1,\,1+\varepsilon^2\right)
.
\end{equation*}

\section{$C^1$-smoothness of $B_{\varepsilon}$.}\label{differentiability}
In the section, we will check that if $\varepsilon\in(0,\,1)$, then $B_{\varepsilon}\in C^1(\Omega_{\varepsilon}\setminus\{(0,\,0)\})$. From~\eqref{equation28} and~\eqref{equation2} we see that $B_{\varepsilon}$ is a piecewise continuously differentiable function (except the point $(0,\,0)$) and with the results of Section~\ref{continuity} we only need to check that the partial derivatives of $B_{\varepsilon}$ are continuous on $\Omega_{\varepsilon}^j\cap\Omega_{\varepsilon}^{j+1}$. Functions $g_{j,\varepsilon}=B_{\varepsilon}|_{\Omega_{\varepsilon}^j}$ are continuously differentiable (except the point $(0,\,0)$) and $g_{j,\varepsilon}=g_{j+1,\varepsilon}$ on $\Omega_{\varepsilon}^j\cap\Omega_{\varepsilon}^{j+1}$. Moreover $\Omega_{\varepsilon}^j\cap\Omega_{\varepsilon}^{j+1}$ consists of distinct non-vertical segments. Since $B_{\varepsilon}$ is linear on these segments, we only need to check that $\frac{\partial g_{j,\varepsilon}}{\partial x_2}=\frac{\partial g_{j+1,\varepsilon}}{\partial x_2}$ on $\Omega_{\varepsilon}^j\cap\Omega_{\varepsilon}^{j+1}$.
\subsection{The case $\varepsilon\in\left(0,\,\frac{1}{2}\right)$.}
First, we calculate
\begin{equation*}
\frac{\partial\beta}{\partial x_2}
=
\frac{\partial \sqrt{\varepsilon^2+x_1^2-x_2}}{\partial x_2}
=
\frac{-1}{2\sqrt{\varepsilon^2+x_1^2-x_2}}
=
\frac{-1}{2\beta}
.
\end{equation*}
Now, we compute$\colon$
\begin{equation*}
\frac{\partial g_{1,\varepsilon}}{\partial x_2}(x_1,\,x_2)
=
\frac{\partial }{\partial x_2}e^{\sqrt{x_2}}
=
\frac{1}{2\sqrt{x_2}}e^{\sqrt{x_2}};
\end{equation*}
\begin{multline}\label{g2dif}
\frac{\partial g_{2,\varepsilon}}{\partial x_2}(x_1,\,x_2)
=
\frac{\partial }{\partial x_2}\left(
\frac{1+\beta}{1+\varepsilon}\exp\{|x_1|-\beta+\varepsilon\}+\frac{\varepsilon-\beta}{1+\varepsilon}\exp\left\{-\frac{1}{\varepsilon}(|x_1|-\beta)+\varepsilon\right\}
\right)
=
\\
=
\exp\{|x_1|-\beta+\varepsilon\}\left(\frac{-1}{(1+\varepsilon)2\beta}+\frac{1+\beta}{(1+\varepsilon)2\beta}\right)
+
\exp\left\{-\frac{1}{\varepsilon}(|x_1|-\beta)+\varepsilon\right\}\left(\frac{1}{(1+\varepsilon)2\beta}-\frac{\varepsilon-\beta}{\varepsilon(1+\varepsilon)2\beta}\right)
=
\\
=
\exp\{|x_1|-\beta+\varepsilon\}\frac{1}{2+2\varepsilon}+\exp\left\{-\frac{1}{\varepsilon}(|x_1|-\beta)+\varepsilon\right\}\frac{1}{2\varepsilon(1+\varepsilon)}
;
\end{multline}
\begin{equation*}
\frac{\partial g_{3,\varepsilon}}{\partial x_2}(x_1,\,x_2)
=
\frac{\partial }{\partial x_2}\left(c(|x_1|-\alpha)+b(x_2-\alpha^2)+e^{\alpha}\right)
=
b;
\end{equation*}
\begin{multline*}
\frac{\partial g_{4,\varepsilon}}{\partial x_2}(x_1,\,x_2)
=
\frac{\partial }{\partial x_2}\left(
\frac{1-\beta}{1-\varepsilon}\exp\{|x_1|+\beta-\varepsilon\}
\right)
=
\exp\{|x_1|+\beta-\varepsilon\}\left(\frac{1}{(1-\varepsilon)2\beta}-\frac{1-\beta}{(1-\varepsilon)2\beta}\right)
=
\\
=
\exp\{|x_1|+\beta-\varepsilon\}\frac{1}{2-2\varepsilon}
.
\end{multline*}
Hence
\begin{equation*}
\frac{\partial g_{2,\varepsilon}}{\partial x_2}\Big|_{\Omega_{\varepsilon}^1\cap\Omega_{\varepsilon}^2}
=
\exp\{\varepsilon\}\frac{1}{2+2\varepsilon}+\exp\left\{\varepsilon\right\}\frac{1}{2\varepsilon(1+1\varepsilon)}
=
\frac{1}{2\varepsilon}e^{\varepsilon}
=
\frac{\partial g_{1,\varepsilon}}{\partial x_2}\Big|_{\Omega_{\varepsilon}^1\cap\Omega_{\varepsilon}^2}
,
\end{equation*}
\begin{equation*}
\frac{\partial g_{2,\varepsilon}}{\partial x_2}\Big|_{\Omega_{\varepsilon}^2\cap\Omega_{\varepsilon}^3}
=
\exp\{\alpha\}\frac{1}{2+2\varepsilon}+\exp\left\{-\frac{\alpha-\varepsilon}{\varepsilon}+\varepsilon\right\}\frac{1}{2\varepsilon(1+\varepsilon)}
,
\end{equation*}
\begin{equation*}
\frac{\partial g_{4,\varepsilon}}{\partial x_2}\Big|_{\Omega_{\varepsilon}^3\cap\Omega_{\varepsilon}^4}
=
\exp\{\alpha\}\frac{1}{2-2\varepsilon}
.
\end{equation*}
Note that
\begin{multline*}
\frac{\frac{\partial g_{2,\varepsilon}}{\partial x_2}\Big|_{\Omega_{\varepsilon}^2\cap\Omega_{\varepsilon}^3}+\frac{\partial g_{4,\varepsilon}}{\partial x_2}\Big|_{\Omega_{\varepsilon}^3\cap\Omega_{\varepsilon}^4}}{2}
=
\frac{\exp\{\alpha\}\frac{1}{2+2\varepsilon}+\exp\left\{-\frac{\alpha-\varepsilon}{\varepsilon}+\varepsilon\right\}\frac{1}{2\varepsilon(1+\varepsilon)}
+\exp\{\alpha\}\frac{1}{2-2\varepsilon}}{2}
=
\\
=
\exp\{\alpha\}\frac{1}{2-2\varepsilon^2}+\exp\left\{-\frac{\alpha-\varepsilon}{\varepsilon}+\varepsilon\right\}\frac{1}{4\varepsilon(1+\varepsilon)}
=
b
=
frac{\partial g_{3,\varepsilon}}{\partial x_2}
.
\end{multline*}
Thus, to check the differentability of $B_{\varepsilon}$, we only need to verify that $\frac{\partial g_{2,\varepsilon}}{\partial x_2}\Big|_{\Omega_{\varepsilon}^2\cap\Omega_{\varepsilon}^3}=\frac{\partial g_{4,\varepsilon}}{\partial x_2}\Big|_{\Omega_{\varepsilon}^3\cap\Omega_{\varepsilon}^4}$ ($\frac{\partial g_{2,\varepsilon}}{\partial x_2}\Big|_{\Omega_{\varepsilon}^2\cap\Omega_{\varepsilon}^3}$ and $\frac{\partial g_{4,\varepsilon}}{\partial x_2}\Big|_{\Omega_{\varepsilon}^3\cap\Omega_{\varepsilon}^4}$ are constant, so we verify equality of constants, not of functions on different domains). The constant $\alpha$ is defined by~\eqref{alpha}, which may be rewritten as
\begin{equation*}
\exp\left\{(\alpha-\varepsilon)\left(\frac{1+\varepsilon}{\varepsilon}\right)\right\}=\frac{1-\varepsilon}{2\varepsilon^2},
\end{equation*}
or, equivalently,
\begin{equation*}
\frac{1}{2\varepsilon(1+\varepsilon)}\exp\left\{-\frac{1}{\varepsilon}(\alpha-\varepsilon)+\varepsilon\right\}
=
\frac{\varepsilon}{(1-\varepsilon)}e^{\alpha}.
\end{equation*}
Therefore,
\begin{equation*}
\frac{\partial g_{2,\varepsilon}}{\partial x_2}\Big|_{\Omega_{\varepsilon}^2\cap\Omega_{\varepsilon}^3}
=
\frac{1}{2(1+\varepsilon)}e^{\alpha}
+
\frac{1}{2\varepsilon(1+\varepsilon)}\exp\left\{-\frac{1}{\varepsilon}(\alpha-\varepsilon)+\varepsilon\right\}
=
\frac{1}{2(1-\varepsilon)}e^{\alpha}
=
\frac{\partial g_{4,\varepsilon}}{\partial x_2}\Big|_{\Omega_{\varepsilon}^3\cap\Omega_{\varepsilon}^4}.
\end{equation*}

\subsection{The case $\varepsilon\in\left[\frac{1}{2},\,1\right)$.}
Firstly, we calculate
\begin{equation*}
\frac{\partial g_{1,\varepsilon}}{\partial x_2}(x_1,\,x_2)
=
\frac{\partial }{\partial x_2}e^{\sqrt{x_2}}
=
\frac{1}{2\sqrt{x_2}}e^{\sqrt{x_2}};
\end{equation*}
\begin{equation*}
\frac{\partial g_{2,\varepsilon}}{\partial x_2}(x_1,\,x_2)
=
\frac{\partial }{\partial x_2}\left(\frac{1-\varepsilon^2+x_2}{2-2\varepsilon}e^{1-\varepsilon}\right)
=
\frac{1}{2-2\varepsilon}e^{1-\varepsilon};
\end{equation*}
\begin{multline*}
\frac{\partial g_{3,\varepsilon}}{\partial x_2}(x_1,\,x_2)
=
\frac{\partial }{\partial x_2}\left(\frac{1-\sqrt{\varepsilon^2+x_1^2-x_2}}{1-\varepsilon}\exp\left(x_1+\sqrt{\varepsilon^2+x_1^2-x_2}-\varepsilon\right)\right)
=
\\
=
\frac{1}{2-2\varepsilon}\exp\left\{|x_1|+\sqrt{\varepsilon+x_1^2-x_2}-\varepsilon\right\}.
\end{multline*}
Hence
\begin{equation*}
\frac{\partial g_{1,\varepsilon}}{\partial x_2}\Big|_{\Omega_{\varepsilon}^1\cap\Omega_{\varepsilon}^2}
=
\frac{1}{2-2\varepsilon}e^{1-\varepsilon}
=
\frac{\partial g_{2,\varepsilon}}{\partial x_2}\Big|_{\Omega_{\varepsilon}^1\cap\Omega_{\varepsilon}^2}
\end{equation*}
and
\begin{equation*}
\frac{\partial g_{3,\varepsilon}}{\partial x_2}\Big|_{\Omega_{\varepsilon}^2\cap\Omega_{\varepsilon}^3}
=
\frac{1}{2-2\varepsilon}e^{1-\varepsilon}
=
\frac{\partial g_{2,\varepsilon}}{\partial x_2}\Big|_{\Omega_{\varepsilon}^2\cap\Omega_{\varepsilon}^3}
.
\end{equation*}

\section{Local concavity of $B_{\varepsilon}$.}\label{concavity}
Cases $\varepsilon=0$ and $\varepsilon\geqslant 1$ are obvious. So, we can assume that $\varepsilon\in(0,\,1)$. Since we have proved $B_{\varepsilon}\in C(\Omega_{\varepsilon})\cap C^1(\Omega_{\varepsilon}\setminus\{(0,\,0)\})$ we only need to check that each $g_{j,\varepsilon}$ is locally concave, and the local concavity of the function on the hole domain will be obtained automatically (see Proposition 3.1.2 of paper~\cite{r6}).

\subsection{The case $\varepsilon\in\left(0,\,\frac{1}{2}\right)$.}
Since $g_{1,\varepsilon}$ depends only on $x_2$, it suffices to verify that $\frac{\partial^2 g_{1,\varepsilon}}{\partial x_2^2}\leqslant 0$. We compute
\begin{equation*}
\frac{\partial^2 g_{1,\varepsilon}}{\partial^2 x_2}(x_1,\,x_2)
=
\frac{\partial}{\partial x_2}\left(\frac{1}{2\sqrt{x_2}}e^{\sqrt{x_2}}\right)
=
\frac{1}{4x_2}e^{\sqrt{x_2}}-\frac{1}{4x_2\sqrt{x_2}}e^{\sqrt{x_2}}
=
\frac{1}{4x_2}e^{\sqrt{x_2}}\left(1-\frac{1}{\sqrt{x_2}}\right)
.
\end{equation*}
Hence $\frac{\partial^2 g_{1,\varepsilon}}{\partial x_2^2}\leqslant 0$ when $x_2\leqslant 1$, in particular, on $\Omega_{\varepsilon}^1$.

To check the local concavity of $g_{j,\varepsilon}$, for $j\geqslant 2$, we only need to check it in $\Omega_{\varepsilon}^j\cap\{x_1\geqslant 0\}$. The function $g_{3,\varepsilon}\big|_{\Omega_{\varepsilon}^3\cap\{x_1\geqslant 0\}}$ is linear and, in particular, locally concave. We have $g_{4,\varepsilon}\big|_{\Omega_{\varepsilon}^4\cap\{x_1\geqslant 0\}}=B_{\varepsilon}^{asym}|_{\Omega_{\varepsilon}^4\cap\{x_1\geqslant 0\}}$, and the local concavity of $B_{\varepsilon}^{asym}$ defined in~\eqref{asymbel} was checked in~\cite{r3}.

So, we only need to verify the local concavity of $g_{2,\varepsilon}\big|_{\Omega_{\varepsilon}^2\cap\{x_1\geqslant 0\}}$. The function $g_{2,\varepsilon}$ is linear along non-vertical segments $l_u=[\left(u-\varepsilon,\,(u-\varepsilon)^2+\varepsilon^2\right),\,\left(u,\,u^2\right)]=\{x_1-\beta=u-\varepsilon\}$. Moreover,
\begin{equation*}
\frac{\partial g_{2,\varepsilon}}{\partial x_2}\Big|_{l_u}
=
e^u\frac{1}{2+2\varepsilon}+\exp\left(-\frac{1}{\varepsilon}(u-\varepsilon)+\varepsilon\right)\frac{1}{2\varepsilon(1+\varepsilon)}
.
\end{equation*}
This means that $d g_{2,\varepsilon}$ is constant along $l_u$. We only need to check that $\frac{\partial^2 g_{2,\varepsilon}}{\partial^2 x_2}\leqslant 0$. We compute
\begin{multline*}
\frac{\partial^2 g_2}{\partial x_2^2}(x_1,\,x_2)
=
\frac{\partial}{\partial x_2}\left(\frac{\partial g_2}{\partial x_2}\right)(x_1,\,x_2)
=
\\
=
\frac{\partial}{\partial x_2}\left(
\frac{1}{2(1+\varepsilon)}\exp\{x_1-\beta+\varepsilon\}
+
\frac{1}{2\varepsilon(1+\varepsilon)}\exp\left\{-\frac{1}{\varepsilon}(x_1-\beta)+\varepsilon\right\}
\right)
=
\\
=
\frac{1}{4\beta(1+\varepsilon)}\exp\{x_1-\beta+\varepsilon\}
-
\frac{1}{4\varepsilon^2\beta(1+\varepsilon)}\exp\left\{-\frac{1}{\varepsilon}(x_1-\beta)+\varepsilon\right\}
.
\end{multline*}
So, the inequality $\frac{\partial^2 g_{2,\varepsilon}}{\partial^2 x_2}\leqslant 0$ is equivalent to
\begin{equation*}
\frac{1}{4\beta(1+\varepsilon)}\exp\{x_1-\beta+\varepsilon\}
\leqslant
\frac{1}{4\varepsilon^2\beta(1+\varepsilon)}\exp\left\{-\frac{1}{\varepsilon}(x_1-\beta)+\varepsilon\right\}
,
\end{equation*}
which reduces to
\begin{equation*}
\exp\left\{(x_1-\beta)\left(\frac{1+\varepsilon}{\varepsilon}\right)\right\}
\leqslant
\frac{1}{\varepsilon^2}
.
\end{equation*}
On the domain $\Omega_{\varepsilon}^2\cap\{x_1\geqslant 0\}$ we have
\begin{equation*}
\exp\left\{(x_1-\beta)\left(\frac{1+\varepsilon}{\varepsilon}\right)\right\}
\leqslant
\exp\left\{(\alpha-\varepsilon)\left(\frac{1+\varepsilon}{\varepsilon}\right)\right\}
=
\frac{1-\varepsilon}{2\varepsilon^2}
\leqslant
\frac{1}{\varepsilon^2}
.
\end{equation*}

\subsection{The case $\varepsilon\in\left[\frac{1}{2},\,1\right)$.}
Since $g_{1,\varepsilon}$ depends only on $x_2$, to verify local concavity of $g_{1,\varepsilon}$, we need only to check, that $\frac{\partial^2 g_{1,\varepsilon}}{\partial x_2^2}\leqslant 0$. We compute
\begin{equation*}
\frac{\partial^2 g_{1,\varepsilon}}{\partial^2 x_2}(x_1,\,x_2)
=
\frac{\partial}{\partial x_2}\left(\frac{1}{2\sqrt{x_2}}e^{\sqrt{x_2}}\right)
=
\frac{1}{4x_2}e^{\sqrt{x_2}}-\frac{1}{4x_2\sqrt{x_2}}e^{\sqrt{x_2}}
=
\frac{1}{4x_2}e^{\sqrt{x_2}}\left(1-\frac{1}{\sqrt{x_2}}\right)
.
\end{equation*}
Hence $\frac{\partial^2 g_{1,\varepsilon}}{\partial x_2^2}\leqslant 0$ when $x_2\leqslant 1$ and in particular in $\Omega_{\varepsilon}^1$.

The function $g_{2,\varepsilon}$ is linear and, in particular, locally concave.

To check the local concavity of $g_{3,\varepsilon}$ we only need to check it in $\Omega_{\varepsilon}^3\cap\{x_1\geqslant 0\}$. There we have $g_{3,\varepsilon}\big|_{\Omega_{\varepsilon}^3\cap\{x_1\geqslant 0\}}=B_{\varepsilon}^{asym}|_{\Omega_{\varepsilon}^3\cap\{x_1\geqslant 0\}}$, and the local concavity of $B_{\varepsilon}^{asym}$ was checked in the work~\cite{r3}.

\section{Optimizers.}\label{attainability}
If $\langle\varphi\rangle_I=x_1$ and $\left\langle\varphi^2\right\rangle_I=x_1^2$, then $\varphi=x_1$ almost everywhere on $I$. Hence $\langle\exp\{|\varphi|\}\rangle_I=\Bel_{\varepsilon}\left(x_1,\,x_1^2\right)=e^{|x_1|}$. So, we need to constructe the functions $\varphi$ only for $\varepsilon>0$ and $(x_1,\,x_2)\in\Omega_{\varepsilon}\setminus P$.
By symmetry, we need to constructe the optimizers only for $(x_1,\,x_2)\in(\Omega_{\varepsilon}\setminus P)\cap\{x_1\geqslant 0\}$.
\subsection{The case $\varepsilon\in\left(0,\,\frac{1}{2}\right]$.}
\subsubsection{Optimizers on $\Omega_{\varepsilon}^1$.}
On $\Omega_{\varepsilon}^1$ we have $x_2\leqslant\varepsilon^2$, so, the segment $\left[\left(-\sqrt{x_2},\,x_2\right),\,\left(\sqrt{x_2},\,x_2\right)\right]$ lies in $\Omega_{\varepsilon}$. Hence the function $\varphi\colon[0,\,1]\to\mathbb{R}$
\begin{equation}
\varphi(t)=
\begin{cases}
\sqrt{x_2},& \text{for }0\leqslant t\leqslant\frac{x_1+\sqrt{x_2}}{2\sqrt{x_2}} ,   \\
-\sqrt{x_2},&\text{for }\frac{x_1+\sqrt{x_2}}{2\sqrt{x_2}}<t\leqslant 1,
\end{cases}
\end{equation}
belongs to $\BMO_{\varepsilon}$. We have $\left(\langle\varphi\rangle_I,\,\left\langle\varphi^2\right\rangle_I\right)=(x_1,\,x_2)$ and 
\begin{equation*}
B_{\varepsilon}\left(x_1,\,x_2\right)
=
e^{\sqrt{x_2}}
=
\left\langle e^{|\varphi|}\right\rangle_I
.
\end{equation*}

\subsubsection{Optimizers on $\Omega_{\varepsilon}^4$.}\label{subsub312}
Since $e^{|x_1|}\geqslant e^{x_1}$ and for $x_1\geqslant 0$, we have $B_{\varepsilon}=B_{\varepsilon}^{asym}$ on $\Omega_{\varepsilon}^4$. Thus, the examples constructed in~\cite{r3} are suitable. For the point $(x_1,\,x_2)\in\Omega_{\varepsilon}\setminus P$ it is $\varphi\colon[0,1]\to \mathbb{R}$,
\begin{equation}
\varphi(t)=
\begin{cases}
\varepsilon\log\left(\frac{h}{t}\right)+s,& \text{for }0\leqslant t\leqslant h,   \\
s,&\text{for }h\leqslant t\leqslant 1,
\end{cases}
\end{equation}
where $h=1-\frac{1}{\varepsilon}\sqrt{\varepsilon^2+x_1^2-x_2}$ and $s=x_1-\varepsilon h$.

\subsubsection{Optimizers on $\Omega_{\varepsilon}^2$.}\label{subsub313}
The method for constructing the optimizers in the cases like this appeared in~\cite{r1}, see Subsection~$7.3$ of that paper. Consider the function $\varphi\colon[0,\,1]\to\mathbb{R}$ given by$\colon$
\begin{equation}\label{funex}
\varphi(t)=
\begin{cases}
u_1-2\varepsilon, & \text{for }t\in\left[0,\,\frac{\mu\nu}{2}\right), \\
u_1, & \text{for }t\in\left[\frac{\mu\nu}{2},\,\mu\nu\right), \\
u+\varepsilon\log\left(\frac{t}{\mu}\right),& \text{for } t\in[\mu\nu,\,\nu),   \\
u,&\text{for }t\in [\nu,\,1),
\end{cases}
\end{equation}
where
\begin{equation}\label{funex1}
u=x_1+\varepsilon-\sqrt{\varepsilon^2-x_2+x_1^2}=x_1+\varepsilon-\beta,\ \ \ \ 
\mu=\frac{u-x_1}{\varepsilon},\ \ \ \ 
\nu=\exp\left(\frac{u_1-u}{\varepsilon}\right),\text{ and }u\geqslant u_1.
\end{equation}
By Lemma~$7.5$ in~\cite{r1}, $\left(\langle\varphi\rangle,\,\left\langle\varphi^2\right\rangle\right)=(x_1,\,x_2)$, $\varphi\in\BMO_{\varepsilon}$ and 
\begin{equation}\label{formula}
\langle f(\varphi)\rangle_{[0,1]}=\frac{1}{\varepsilon}e^{-\frac{u}{\varepsilon}}
\left(
\frac{f(u_1)-f(u_1-2\varepsilon)}{2}e^{\frac{u_1}{\varepsilon}}+
\int\limits_{u_1}^{u}f'(s)e^{\frac{s}{\varepsilon}}\,ds
\right)(x_1-u)
+f(u)
,
\end{equation}
for any reasonable function $f$.
In our case we substitute $u_1=\varepsilon$ and $f(t)=e^{|t|}$. Then we get 
\begin{multline*}
\left\langle e^{|\varphi|}\right\rangle_{[0,1]}
=
\frac{1}{\varepsilon}e^{-\frac{u}{\varepsilon}}
\left(
\int\limits_{\varepsilon}^{u}e^{s\frac{1+\varepsilon}{\varepsilon}}\,ds
\right)(x_1-u)
+e^u
=
\frac{1}{\varepsilon}e^{-\frac{u}{\varepsilon}}
\frac{e^{s\frac{1+\varepsilon}{\varepsilon}}}{\frac{1+\varepsilon}{\varepsilon}}\Bigg|_{\varepsilon}^u
(x_1-u)
+e^u
=
\\
=
(x_1-u)e^{-\frac{u}{\varepsilon}}
\left(\frac{e^{u\frac{1+\varepsilon}{\varepsilon}}}{1+\varepsilon}-\frac{e^{1+\varepsilon}}{1+\varepsilon}\right)
+e^u
=
e^u\left(1+\frac{x_1-u}{1+\varepsilon}\right)-e^{1+\varepsilon-\frac{u}{\varepsilon}}\frac{x_1-u}{1+\varepsilon}
=
\\
=
\exp(x_1+\varepsilon-\beta)\frac{1+\beta}{1+\varepsilon}-\exp\left(-\frac{x_1+\varepsilon-\beta}{\varepsilon}+1+\varepsilon\right)\frac{\beta-\varepsilon}{1+\varepsilon}
=
B_{\varepsilon}(x_1,\,x_2)
.
\end{multline*}

\subsubsection{Optimizers on $\Omega_{\varepsilon}^3$.}
Let $x=(x_1,\,x_2)\in\Omega_{\varepsilon}^3$. Let $l$ be a line that contains $x$ and does not intersect the set 
\begin{equation*}
\left\{(x_1,\,x_2)\in\mathbb{R}^2\big|\,x_2>x_1^2+\varepsilon^2\right\}
.
\end{equation*}
Then $l$ intersects the segments $\left[\left(\alpha,\,\alpha^2\right),\,\left(\alpha-\varepsilon,\,(\alpha-\varepsilon)^2+\varepsilon^2\right)\right]$ and $\left[\left(\alpha,\,\alpha^2\right),\,\left(\alpha+\varepsilon,\,(\alpha+\varepsilon)^2+\varepsilon^2\right)\right]$. Let those intersection points be $y$ and $z$ respectively. Then, $x=\gamma y+(1-\gamma)z$ for some $\gamma\in[0,\,1]$. By Subsection~\ref{subsub312}, we have a non-increasing function $\varphi_y\in\BMO_{\varepsilon}$ such that 
\begin{equation*}
\left(\langle\varphi_y\rangle_{[0,1]},\,\langle\varphi_y^2\rangle_{[0,1]}\right)=y,\ \ B_{\varepsilon}(y)=\left\langle e^{|\varphi_y|}\right\rangle_{[0,1]},\text{ and }\varphi_y(1)=\alpha
.
\end{equation*}
By Subsection~\ref{subsub313}, we have non-decreasing function $\varphi_z\in\BMO_{\varepsilon}$ such that 
\begin{equation*}
\left(\langle\varphi_z\rangle_{[0,1]},\,\langle\varphi_z^2\rangle_{[0,1]}\right)=z,\ \ B_{\varepsilon}(z)=\left\langle e^{|\varphi_z|}\right\rangle_{[0,1]},\text{ and }\varphi_z(1)=\alpha
.
\end{equation*}
Let $\varphi\colon[0,\,1]\to\mathbb{R}$ be defined as follows.
\begin{equation*}
\varphi(t)=
\begin{cases}
\varphi_z\left(\dfrac{t}{1-\gamma}\right),&\text{for }t\in[0,\,1-\gamma],\\
\varphi_y\left(\dfrac{1-t}{\gamma}\right),&\text{for }t\in(1-\gamma,\,1].
\end{cases}
\end{equation*}
Then $\left(\langle\varphi\rangle_{[0,1]},\,\langle\varphi^2\rangle_{[0,1]}\right)=\gamma y+(1-\gamma)z=x$ and from the linearity of the function $B_{\varepsilon}\big|_{\Omega_{\varepsilon}^3}$ we have
\begin{equation*}
\left\langle e^{|\varphi|}\right\rangle_{[0,1]}
=
\gamma\left\langle e^{|\varphi_y|}\right\rangle_{[0,1]}
+
(1-\gamma)\left\langle e^{|\varphi_z|}\right\rangle_{[0,1]}
=
\gamma B_{\varepsilon}(y)+(1-\gamma)B_{\varepsilon}(z)
=
B_{\varepsilon}(x)
.
\end{equation*}
What is more, $\varphi$ is non-decreasing. Since $[y,\,z]\subseteq \Omega_{\varepsilon}$ it follows from Corollaries $3.12$ and $3.13$ of paper~\cite{r5} that $\varphi\in\BMO_{\varepsilon}$.

\subsection{The case $\varepsilon\in\left[\frac{1}{2},\,1\right)$.}
\subsubsection{Optimizers on $\Omega_{\varepsilon}^1$.}\label{subsub321}
On $\Omega_{\varepsilon}^1$ we have $x_2\leqslant(1-\varepsilon)^2\leqslant\varepsilon^2$, so, the segment $\left[\left(-\sqrt{x_2},\,x_2\right),\,\left(\sqrt{x_2},\,x_2\right)\right]$ lies in $\Omega_{\varepsilon}$. Hence the function $\varphi\colon[0,\,1]\to\mathbb{R}$
\begin{equation}
\varphi(t)=
\begin{cases}
\sqrt{x_2},& \text{for }0\leqslant t\leqslant\frac{x_1+\sqrt{x_2}}{2\sqrt{x_2}} ,   \\
-\sqrt{x_2},&\text{for }\frac{x_1+\sqrt{x_2}}{2\sqrt{x_2}}<t\leqslant 1.   
\end{cases}
\end{equation}
belongs to $\BMO_{\varepsilon}$, and satisfies
\begin{equation*}
\left(\langle\varphi\rangle_I,\,\langle\varphi^2\rangle_I\right)=(x_1,\,x_2),
\end{equation*}
\begin{equation*}
B_{\varepsilon}\left(x_1,\,x_2\right)
=
e^{\sqrt{x_2}}
=
\left\langle e^{|\varphi|}\right\rangle_I
.
\end{equation*}

\subsubsection{Optimizers on $\Omega_{\varepsilon}^3$.}\label{subsub322}
Since $e^{|x|}\geqslant e^x$ and for $x_1\geqslant 0$ we have $B_{\varepsilon}=B_{\varepsilon}^{asym}$ on $\Omega_{\varepsilon}^3$, the examples constructed in~\cite{r3} are suitable. For the point $(x_1,\,x_2)\in\Omega_{\varepsilon}\setminus P$ it is $\varphi\colon[0,1]\to \mathbb{R}$,
\begin{equation}
\varphi(t)=
\begin{cases}
\varepsilon\log\left(\frac{h}{t}\right)+s,& \text{for }0\leqslant t\leqslant h,   \\
s,&\text{for }h\leqslant t\leqslant 1,
\end{cases}
\end{equation}
where $s=1-\frac{1}{\varepsilon}\sqrt{\varepsilon^2+x_1^2-x_2}$ and $h=x_1-\varepsilon s$.

\subsubsection{Optimizers on $\Omega_{\varepsilon}^2$.}
Let $x=(x_1,\,x_2)\in\Omega_{\varepsilon}^2$. Then let $l$ be a line that contains $x$ and does not intersect the set 
\begin{equation*}
\left\{(x_1,\,x_2)\in\mathbb{R}^2\big|\,x_2>x_1^2+\varepsilon^2\right\}
.
\end{equation*}
In this case, $l\cap\Omega_{\varepsilon}^2$ is a segment, let it be $[y,\,z]$. Then, $x=\gamma y+(1-\gamma)z$ for some $\gamma\in[0,\,1]$. By the results of Subsections~\ref{subsub321} and~\ref{subsub322}, we have functions $\varphi_y,\varphi_z\in\BMO_{\varepsilon}$, such that 
\begin{equation*}
\left(\langle\varphi_y\rangle_{[0,1]},\,\langle\varphi_y^2\rangle_{[0,1]}\right)=y,\ \ B_{\varepsilon}(y)=\left\langle e^{|\varphi_y|}\right\rangle_{[0,1]},\ \ \left(\langle\varphi_z\rangle_{[0,1]},\,\langle\varphi_z^2\rangle_{[0,1]}\right)=z,\ \ B_{\varepsilon}(z)=\left\langle e^{|\varphi_z|}\right\rangle_{[0,1]}\text{ and }\varphi_z(1)=\alpha
.
\end{equation*}
Let $\tilde\varphi\colon[0,\,1]\to\mathbb{R}$ be defined as follows$\colon$
\begin{equation*}
\tilde\varphi(t)=
\begin{cases}
\varphi_z\left(\dfrac{t}{1-\gamma}\right),&\text{for }t\in[0,\,1-\gamma],\\
\varphi_y\left(\dfrac{1-t}{\gamma}\right),&\text{for }t\in(1-\gamma,\,1].
\end{cases}
\end{equation*}
Let $\varphi$ be the non-decreasing rearrangement (equimeasurable non-decreasing function) of $\tilde\varphi$. Then $\left(\langle\varphi\rangle_{[0,1]},\,\langle\varphi^2\rangle_{[0,1]}\right)=\gamma y+(1-\gamma)z=x$ and from the linearity of the function $B_{\varepsilon}\big|_{\Omega_{\varepsilon}^3}$ we have:
\begin{equation*}
\left\langle e^{|\varphi|}\right\rangle_{[0,1]}
=
\gamma\left\langle e^{|\varphi_y|}\right\rangle_{[0,1]}
+
(1-\gamma)\left\langle e^{|\varphi_z|}\right\rangle_{[0,1]}
=
\gamma B_{\varepsilon}(y)+(1-\gamma)B_{\varepsilon}(z)
=
B_{\varepsilon}(x)
.
\end{equation*}
Since $[y,\,z]\subseteq \Omega_{\varepsilon}$, it follows from Corollaries $3.12$ and $3.13$ of~\cite{r5} that $\varphi\in\BMO_{\varepsilon}$.

For example, for the point $\left(0,\,\varepsilon^2\right)$, we get the following function$\colon$
\begin{equation*}
\varphi(t)=
\begin{cases}
-\varepsilon\log\left(\frac{2\varepsilon-1}{4\varepsilon t}\right)-1+\varepsilon, & \text{for }t\in\left[0,\,\frac{2\varepsilon-1}{4\varepsilon}\right],\\
-1+\varepsilon,&\text{for }t\in\left[\frac{2\varepsilon-1}{4\varepsilon},\,\frac{1}{2}\right],\\
1-\varepsilon,&\text{for }t\in\left(\frac{1}{2},\,\frac{2\varepsilon+1}{4\varepsilon}\right],\\
\varepsilon\log\left(\frac{2\varepsilon-1}{4\varepsilon (1-t)}\right)+1-\varepsilon, & \text{for }t\in\left[\frac{2\varepsilon+1}{4\varepsilon},\,1\right].
\end{cases}
\end{equation*}

\subsection{The case $\varepsilon\geqslant 1$.}
Since $e^{|x|}\geqslant e^x$ and $B_{\varepsilon}=B_{\varepsilon}^{asym}$ in the case, the examples constructed in~\cite{r3} are suitable. For the point $(x_1,\,x_2)\in\Omega_{\varepsilon}\setminus P$ it is $\varphi\colon[0,1]\to \mathbb{R}$,
\begin{equation}
\varphi(t)=
\begin{cases}
\varepsilon\log\left(\frac{a}{t}\right)+b,& \text{for }0\leqslant t\leqslant a,   \\
b,&\text{for }a\leqslant t\leqslant 1,
\end{cases}
\end{equation}
where $a=1-\frac{1}{\varepsilon}\sqrt{\varepsilon^2+x_1^2-x_2}$ and $b=x_1-\varepsilon a$.


\newpage
\begin{thebibliography}{99}

%\bibitem{r11}D. L. Burkholder, \emph{Boundary value problems and sharp inequalities for martingale transforms}, Ann.
%Prob. \textbf{12} (1984), no. 3, 647–702.
%
%\bibitem{r10}L. Caffarelli, L. Nirenberg, J. Spruck \emph{The Dirichlet Problem for the Degenerate Monge-Ampère Equation}, RMI Volume 2, Issue 1, 1986, pp. 19–27.
%
%\bibitem{r9} Guan, Bo \emph{The Dirichlet problem for Monge-Ampère equations in non-convex domains and spacelike hypersurfaces of constant Gauss curvature}, Trans. Amer. Math. Soc. \textbf{350} (1998), no. 12, 4955–4971.

\bibitem{r7}P. Ivanishvili, N. N. Osipov, D. M. Stolyarov, V. I. Vasyunin, and P. B. Zatitskiy, \emph{Bellman function
for extremal problems in BMO}, Trans. Amer. Math. Soc. \textbf{368} (2016), 3415–3468.

\bibitem{r6} P. Ivanishvili, D. M. Stolyarov, V. I. Vasyunin, and P. B. Zatitskiy, \emph{Bellman function for extremal
problems on $\BMO$ II: evolution}, Mem. Amer. Math. Soc. \textbf{255} (2018), no. 1220.


\bibitem{r4} F. John, L. Nirenberg. \emph{On functions of bounded mean oscillation}, Comm. Pure Appl. Math., Vol. \textbf{14}
(1961), pp. 415-426.

\bibitem{Kor} A. A. Korenovskii, \emph{On the connection between mean oscillation and exact integrability classes of functions}, Mat. Sb. {\bf 181}:12 (1990), 1721--1727; translated in Math. USSR-Sb. {\bf 71}:2 (1992), 561--567.

%\bibitem{r8}Krylov, N. V. (1990). \emph{Smoothness of the payoff function for a controllable processin a domain}, Math. USSR-Izv. \textbf{34}:65–95.

\bibitem{Lerner} A. Lerner, \emph{The John--Nirenberg inequality with sharp constants}, Comptes Rendus Mathemathique {\bf 351}:11-12 (2013), 463--466.

%\bibitem{r13}F. L. Nazarov and S. R. Treil, \emph{The hunt for a Bellman function: applications to estimates for singular
%integral operators and to other classical problems of harmonic analysis}, Algebra i Analiz \textbf{8} (1996),
%no. 5, 32–162, translation in St. Petersburg Math. J. 8 (1997), no. 5, 721–824.

\bibitem{r12}A. Osekowski, \emph{Sharp martingale and semimartingale inequalities}, Monografie Matematyczne IMPAN \textbf{72}, Springer Basel, 2012.

\bibitem{Slavin15} L. Slavin, \emph{The John--Nirenberg constant of $\BMO^p$, $1 \leq p \leq 2$}, https://arxiv.org/abs/1506.04969.

\bibitem{r1}L. Slavin and V. Vasyunin, \emph{Sharp $\text{L}^p$ estimates on $\BMO$},  Ind. Univ. Math. J. \textbf{61} (2012), 1051–1110.

\bibitem{r3}L. Slavin, V. Vasyunin, \emph{Sharp results in the integral-form John-Nirenberg inequality}, Trans. Amer.
Math. Soc. \textbf{363}:8 (2011), 4135–4169.

\bibitem{SlavinVasyunin} L. Slavin, V. Vasyunin, \emph{The John--Nirenberg constant of $\BMO^p$, $p > 2$}, Algebra i Analiz {\bf 28}:2 (2016), 72--96 (in Russian), tranlsated in St. Petersburg Mathematical Journal {\bf 28} (2017), 181--196.


\bibitem{r5}D. M. Stolyarov and P. B. Zatitskiy, \emph{Theory of locally concave functions and its applications to sharp
estimates of integral functionals}, Adv. Math. \textbf{291} (2016).

\bibitem{StolyarovZatitskiy} D. Stolyarov and P. Zatitskiy, \emph{Sharp transference principle for $\BMO$ and $A_p$}, J. Funct. Anal. {\bf 281} (2021),
no. 6, 109085.

\bibitem{r2}V. Vasyunin, A. Volberg, \emph{Sharp constants in the classical weak form
of the John–Nirenberg inequality}, Proc. Lond. Math. Soc. \textbf{108}:6 (2014), 1417–1434.

\bibitem{r14}V. Vasyunin and A. Volberg, \emph{The Bellman function technique in Harmonic Analysis}, Cambridge University Press, 2020.



\end{thebibliography}
\end{document}